\documentclass{amsart}
\usepackage{amsmath}
\usepackage{amssymb}
\usepackage{amsfonts}
\usepackage{mathrsfs}
\usepackage{latexsym}
\usepackage{amsbsy}
\usepackage[all]{xy}
\usepackage{xypic}
\usepackage{amscd}
\usepackage[dvips]{graphicx}

\newtheorem{Thm}{Theorem}[section]

\newtheorem{Lem}[Thm]{Lemma}

\theoremstyle{remark}

\def\Ext{\text{Ext}}
\def\Hom{\text{Hom}}

\def\Aut{\text{Aut}}
\def\dim{\text{dim}}

\def\dimv{\underline{\text{dim}}}
\def\rep{\text{rep}}
\def\modl{\text{mod}}
\def\ad{\text{ad}}
\def\ind{\text{ind}}

\def\rep{\text{rep}}

\def\mfk{\mathfrak}
\def\mfkg{\mathfrak{g}}

\def\mc{\mathcal}

\def\mci{\mc{I}}

\def\mo{\mathscr{O}}
\def\Coh{\text{Coh}}

\def\bbc{{\mathbb C}}
\def\bbz{{\mathbb Z}}
\begin{document}

\title{complex Lie algebras corresponding to weighted projective lines}

\author{Rujing Dou}
\address{Department of Mathematics, Tsinghua University, Beijing 100084, P.R.China}
\email{drj05@mails.thu.edu.cn}

\author{Jie Sheng}
\address{Academy of Mathematics and Systems Science, Chinese Academy of
Science, Beijing 100190, P.R.China} \email{shengjie@amss.ac.cn}

\author{Jie Xiao}
\address{Department of Mathematics, Tsinghua University, Beijing 100084, P.R.China}
\email{jxiao@math.tsinghua.edu.cn}

\thanks{
\\
The research was supported in part by NSF of China (No. 10631010)
and by NKBRPC (No. 2006CB805905) }


\bigskip

\begin{abstract}
The aim of this paper is to give an alternative proof of Kac's
theorem for weighted projective lines (\cite{W}) over the complex
field. The geometric realization of complex Lie algebras arising
from derived categories (\cite{XXZ}) is essentially used.
\end{abstract}

\maketitle
\section{Introduction}
It is well known that the dimension vectors of indecomposable
representations of quiver $Q$ correspond $1-1$ to the positive roots
of the Kac-Moody algebra associated to $Q$.

In \cite{W}, Crawley-Boevey proved an analogue of Kac's Theorem as
follows:
 \begin{Thm}
If $\mathbb{X}_{\mathbf{p},\underline{\lambda}}$ is a weighted
projective line over an algebraically closed field $K$ and
$\alpha\in\hat{Q}$, there is an indecomposable sheaf in
$\Coh(\mathbb{X}_{\mathbf{p},\underline{\lambda}})$ of type $\alpha$
if and only if $\alpha$ is a positive root. Moreover, there is a
unique indecomposable for a real root, infinitely many for an
imaginary root.
\end{Thm}

This theorem describes the possible dimension vectors of
indecomposable sheaves. In order to prove it, Crawley-Boevey reduced
to the case when $K$ is the algebraic closure of a finite field. He
worked over a finite field $F_q$ and associated a Lie algebra $L$ to
the category of coherent sheaves on a weighted projective line over
this finite field. We note that the Lie algebra $L$ is defined over
a field $F$, which has characteristic $l$ such that $q=1$ in $F$.

We find that the proof can be simplified when $K$ is changed to the
complex field $\mathbb{C}$. Using \cite{XXZ} and the derived
equivalence between the category of coherent sheaves on a weighted
projective line and the module category of the corresponding
canonical algebra, we construct a Lie algebra $L$ on the category of
coherent sheaves on a weighted projective line over $\mathbb{C}$ and
find elements which satisfy the relations of the loop algebra. We
calculate the Euler characteristics instead of counting numbers.

Let $v$ be a vertex of the star-shaped graph (see \ref{star-shaped
graph}) and write $\alpha_v$ for the simple root corresponding to
$v$. Let $e\in L_{\alpha_v}$, $f\in L_{-\alpha_v}$, using the
standard arguments in Lie algebra over the base field $\mathbb{C}$,
we have the isomorphism $L_{\phi}\simeq L_{s_v(\phi)}$, i.e, the
simple reflection induces isomorphism. Finally, we reduce to three
simple cases by a sequence of reflections which were solved in
\cite{W2}.

We note that in the process of the proof of the Kac Theorem on
weighted projective lines, the operator $\theta=exp(\ad\ e)exp(-\ad\
f)exp(\ad\ e)$ in the $sl_2$-representation can be defined directly
and the definition trouble occurring in the case of the finite field
is avoided. Moreover, the process of finding a suitable field as the
base field of the Lie algebra can be omitted. This simplifies the
proof.

\section{Lie algebras arising from derived categories}
\subsection{}Let $\Lambda$ be a finite dimensional and finite global
dimensional associative algebra over $\mathbb{C}$. We can write (up
to Morita equivalence) $\Lambda=\bbc Q/J$, where $Q$ is a quiver and
$J$ is the admissible ideal generated by a set $R$ of relations.

Consider the category $\modl\Lambda$ of finite dimensional
$\Lambda$-modules and its bounded derived category $D^{b}(\Lambda)$.
In \cite{XXZ}, Xiao, Xu and Zhang obtained a geometric realization
of complex Lie algebras arising from the root category
$D^{b}(\Lambda)/(T^{2})$. We will give a short review here.

\subsection{}
We fix $\{P_{1},P_{2},\cdots,P_{l}\}$ to be a complete set of
indecomposable projective $\Lambda$-modules. A complex $C^{\bullet}$
of $\Lambda$-modules is called a period-$2$ complex if it satisfies
$C^{\bullet}[2]=C^{\bullet}$. Let
$P^{\bullet}=(P^{0},P^{1},\partial_{0},\partial_{1})$ be a
period-$2$ complex of projective $\Lambda$-modules such that each
$P^{i}$ has the decomposition
$P^{i}=\bigoplus^{l}_{j=1}e^{i}_{j}P_{j}$. We denote by
$\underline{e}(P^{i})$ the vector
$(e^{i}_{1},e^{i}_{2},\cdots,e^{i}_{l})$, then
$\underline{\mathbf{e}}=(\underline{e}(P^{0}),\underline{e}(P^{1}))$
is called the projective dimension sequence of $P^{\bullet}$. We
define $\mathcal{P}_{2}(\Lambda,\underline{\mathbf{e}})$ to be the
subset of $$\Hom_{\Lambda}(P^{0},P^{1})\times
\Hom_{\Lambda}(P^{1},P^{0})$$ which consists of
$(\partial_{0},\partial_{1})$ such that $\partial_{0}\partial_{1}=0$
and $\partial_{1}\partial_{0}=0$.

The algebraic group
$G_{\underline{\mathbf{e}}}=\Aut_{\Lambda}(P^{0})\times
\Aut_{\Lambda}(P^{1})$ acts on
$\mathcal{P}_{2}(\Lambda,\underline{\mathbf{e}})$ by conjugation
action. Thus two projective complexes in
$\mathcal{P}_{2}(\Lambda,\underline{\mathbf{e}})$ are in the same
orbit under the $G_{\underline{\mathbf{e}}}$-action if and only if
they are quasi-isomorphic.

Let $K_{0}$ be the Grothendieck group of $D^{b}(\Lambda)$, also of
$D^{b}(\Lambda)/(T^{2})$. There is a canonical surjection from the
abelian group of projective dimension sequences to $K_{0}$, which
will be denoted by $\dimv$. We define
$\mathcal{P}_{2}(\Lambda,\mathbf{d})=\bigcup_{\underline{\mathbf{e}}\in
\dimv^{-1}(\mathbf{d})}\mathcal{P}_{2}(\Lambda,\underline{\mathbf{e}})$
for any $\mathbf{d}\in K_{0}$. Then
$\mathcal{P}_{2}(\Lambda,\mathbf{d})$ has a natural topological
structure induced by that of
$\mathcal{P}_{2}(\Lambda,\underline{\mathbf{e}})$, see \cite{XXZ}
for details. Thus $G_{\mathbf{d}}=\bigcup_{\underline{\mathbf{e}}\in
\dimv^{-1}(\mathbf{d})}G_{\underline{\mathbf{e}}}$ partially acts on
$\mathcal{P}_{2}(\Lambda,\mathbf{d})$. Moreover, we set
$$T_{\underline{\mathbf{e}}}=\{t^{\pm}_{x}|x\in
\mathcal{P}_{2}(\Lambda,\underline{\mathbf{e}})\ \text{is
constructible}\}$$ and $T=\bigcup_{\underline{\mathbf{e}}\in
\dimv^{-1}(\mathbf{0})}T_{\underline{\mathbf{e}}}$ whose action on
$\mathcal{P}_{2}(\Lambda,\mathbf{d})$ is also partially defined.
With the groupoid $\langle G_{\mathbf{d}},T\rangle$ acting on
$\mathcal{P}_{2}(\Lambda,\mathbf{d})$, we have that
$$\mathcal{QP}_{2}(\Lambda,\mathbf{d})=\mathcal{P}_{2}(\Lambda,\mathbf{d})/\sim=\mathcal{P}_{2}(\Lambda,\mathbf{d})/\langle G_{\mathbf{d}},T\rangle$$
where $x\sim y$ in $\mathcal{P}_{2}(\Lambda,\mathbf{d})$ if and only
if their corresponding complexes are quasi-isomorphic.

\subsection{}\label{lie algebra}
We denote by $M(X)$ the set of all constructible functions on an
algebraic variety $X$ with values in $\bbc$. The set $M(X)$ is
naturally a $\bbc$-linear space. Let $G$ be an algebraic group
acting on $X$. Then we denote by $M_{G}(X)$ the subspace of $M(X)$
consisting of all $G$-invariant functions.

Let $\mathbf{d}$ be a dimension vector in $K_{0}$ and $\mathcal{O}$
be a $\langle G_{\mathbf{d}},T\rangle$-invariant and support-bounded
constructible subset of $\mathcal{P}_{2}(\Lambda,\mathbf{d})$. Here
support-bounded means there exists a projective dimension sequence
$\underline{\mathbf{e}}$ such that $\mathcal{O}=\langle
G_{\mathbf{d}},T\rangle(\mathcal{O}\cap
\mathcal{P}_{2}(\Lambda,\underline{\mathbf{e}}))$ and
$\underline{\mathbf{e}}$ is called a support projective dimension
sequence of $\mathcal{O}$.

We define the function $1_{\mathcal{O}}:
\mathcal{P}_{2}(\Lambda,\mathbf{d})\rightarrow \bbc$ given by taking
values $1$ on each point in  $\mathcal{O}$ and $0$ otherwise. A
function $f$ on $\mathcal{P}_{2}(\Lambda,\mathbf{d})$ is called
$\langle G_{\mathbf{d}},T\rangle$-invariant constructible function
if $f$ can be written as a sum of finite sums
$\sum_{i}m_{i}1_{\mathcal{O}_{i}}$ where $m_{i}\in \bbc$ and any
$\mathcal{O}_{i}$ is $\langle G_{\mathbf{d}},T\rangle$-invariant and
support-bounded constructible subset of
$\mathcal{P}_{2}(\Lambda,\mathbf{d})$. Let
$\underline{\mathbf{e}}_{1}$ and $\underline{\mathbf{e}}_{2}$ be
projective dimension sequences in $\dimv^{-1}(\mathbf{d})$. Two
constructible functions $f_{i}\in
M_{G_{\underline{\mathbf{e}}_{i}}}(\mathcal{P}_{2}(\Lambda,\underline{\mathbf{e}}_{i}))$,
$i=1,2$ are equivalent if there exists a $\langle
G_{\mathbf{d}},T\rangle$-invariant constructible $F$ over
$\mathcal{P}_{2}(\Lambda,\mathbf{d})$ such that
$f_{i}=F|_{\mathcal{P}_{2}(\Lambda,\underline{\mathbf{e}}_{i})}$,
$i=1,2$. Let $f\in
M_{G_{\underline{\mathbf{e}}}}(\Lambda,\underline{\mathbf{e}})$ and
$\underline{\mathbf{e}}\in \dimv^{-1}(\mathbf{d})$. The equivalent
class of $f$ is denoted by $\hat{f}$. Let
$M_{GT}(\mathcal{P}_{2}(\Lambda,\mathbf{d}))$ be the space of the
equivalence classes $\hat{f}$ of constructible functions $f$ over
$\mathcal{P}_{2}(\Lambda,\underline{\mathbf{e}})$ for any
$\underline{\mathbf{e}}\in \dimv^{-1}(\mathbf{d})$.

An equivalence class $\hat{f}\in
M_{GT}(\mathcal{P}_{2}(\Lambda,\mathbf{d}))$ is called
indecomposable if any point in $supp(f)$ is indecomposable in the
(relative) homotopy category of all period-$2$ complexes of
projective modules. Let $I_{GT}(\mathbf{d})$ be the $\bbc$-space of
all indecomposable equivalence classes in
$M_{GT}(\mathcal{P}_{2}(\Lambda,\mathbf{d}))$.

Let $\mathcal{O}_{1}\subset
\mathcal{P}_{2}(\Lambda,\underline{\mathbf{e}}'')\subset
\mathcal{P}_{2}(\Lambda,\mathbf{d}_{1})$ and $\mathcal{O}_{2}\subset
\mathcal{P}_{2}(\Lambda,\underline{\mathbf{e}}')\subset
\mathcal{P}_{2}(\Lambda,\mathbf{d}_{2})$ be
$G_{\underline{\mathbf{e}}''}$- and
$G_{\underline{\mathbf{e}}'}$-invariant constructible set,
respectively. For $L\in
\mathcal{P}_{2}(\Lambda,\underline{\mathbf{e}}'+\underline{\mathbf{e}}'')$,
we set
$$W(\mathcal{O}_{1},\mathcal{O}_{2};L)=\{(f,g,h)|Y \xrightarrow{f}L\xrightarrow{g}X\xrightarrow{h}Y[1] \ \text{is a distinguished
triangle}$$ $$\text{with}\  X\in \mathcal{O}_{1}, Y\in
\mathcal{O}_{2}\},$$ then the quotient space
$W(\mathcal{O}_{1},\mathcal{O}_{2};L)/G_{\underline{\mathbf{e}}''}\times
G_{\underline{\mathbf{e}}'}$ is independent of choices of support
projective dimension sequences of both $\langle
G_{\mathbf{d}_{1}},T\rangle \mathcal{O}_{1}$ and $\langle
G_{\mathbf{d}_{2}},T\rangle \mathcal{O}_{2}$. So we denote it by
$V(\mathcal{O}_{1},\mathcal{O}_{2};L)$.

Thus the convolution multiplication
$\hat{1}_{\mathcal{O}_{1}}*\hat{1}_{\mathcal{O}_{2}}\in
M_{GT}(\mathcal{P}_{2}(\Lambda,\mathbf{d}_{1}+\mathbf{d}_{2}))$ can
be defined as follows:
$$\hat{1}_{\mathcal{O}_{1}}*\hat{1}_{\mathcal{O}_{2}}(L)=F^{L}_{\mathcal{O}_{1}\mathcal{O}_{2}}:=\chi(V(\mathcal{O}_{1},\mathcal{O}_{2};L))$$
where $\chi$ denotes the quasi Euler characteristic of quotient
space as in \cite{XXZ}.

We set $\mfk{n}=\bigoplus_{d\in K_{0}}I_{GT}(\mathbf{d})$ and
$\mfk{h}=K_{0}\otimes_{\bbz}\bbc$ which is spanned by
$\{h_{\mathbf{d}}|\mathbf{d}\in K_{0}\}$. The symmetric Euler
bilinear form on $\mfk{h}$ is given as
$$(h_{\mathbf{d}_{1}}|h_{\mathbf{d}_{2}})=\dim_{\bbc}\Hom(X,Y)-\dim_{\bbc}\Hom(X,Y[1])$$
$$+\dim_{\bbc}\Hom(Y,X)-\dim_{\bbc}\Hom(Y,X[1])$$
for any $X\in \mathcal{P}_{2}(\Lambda,\mathbf{d_{1}})$, $Y\in
\mathcal{P}_{2}(\Lambda,\mathbf{d_{2}})$.

Thus $\mfk{g}=\mfk{h}\oplus \mfk{n}$ becomes a Lie algebra over
$\bbc$ with the Lie bracket $[-,-]$ defined below.
$$[\hat{1}_{\mathcal{O}_{1}},\hat{1}_{\mathcal{O}_{2}}]=[\hat{1}_{\mathcal{O}_{1}},\hat{1}_{\mathcal{O}_{2}}]_{\mfk{n}}+\chi(\overline{\mathcal{O}_{1}\cap \mathcal{O}_{2}[1]})h_{\mathbf{d}_{1}}$$
where $\overline{\mathcal{O}_{1}\cap \mathcal{O}_{2}[1]}\backsimeq
(\mathcal{O}_{1}\cap
\mathcal{O}_{2}[1])_{\underline{\mathbf{e}}}/G_{\underline{\mathbf{e}}}$
for a support projective dimension sequence of $\mathcal{O}_{1}\cap
\mathcal{O}_{2}[1]$.
$$[\hat{1}_{\mathcal{O}_{1}},\hat{1}_{\mathcal{O}_{2}}]_{\mfk{n}}(L):= F^{L}_{\mathcal{O}_{1}\mathcal{O}_{2}}-F^{L}_{\mathcal{O}_{2}\mathcal{O}_{1}}$$
$$[h_{\mathbf{d}_{1}},\hat{1}_{\mathcal{O}_{2}}]:=(h_{\mathbf{d}_{1}}|h_{\mathbf{d}_{2}})\hat{1}_{\mathcal{O}_{2}},\
[\hat{1}_{\mathcal{O}_{2}},h_{\mathbf{d}_{1}}]:=-(h_{\mathbf{d}_{1}}|h_{\mathbf{d}_{2}})\hat{1}_{\mathcal{O}_{2}}$$
$$[h_{\mathbf{d}_{1}},h_{\mathbf{d}_{2}}]:=0.$$

\section{The category of coherent sheaves on weighted projective
lines}\label{coh sheaf on w p l}
\subsection{Weighted projective lines}\label{weighted projective line}
Let $\mathbf{p}=(p_{1},p_{2},\cdots,p_{n})\in(\mathbb{N}^{*})^{n}$
and $\underline{\lambda}=\{\lambda_{1},\cdots,\lambda_{n}\}$ be a
collection of distinct closed points  on the projective line
$\mathbb{P}^{1}(\mathbb{C})$.

Instead of giving the definition, we give a description of the
structure of the category
$\Coh(\mathbb{X}_{\mathbf{p},\underline{\lambda}})$ (see ~\cite{GL}
for details).

Let $\mathscr{F}$ and $\mathscr{T}$ be two full extension-closed
subcategories of
$\Coh(\mathbb{X}_{\mathbf{p},\underline{\lambda}})$. For any sheaf
$\mathscr{M}\in\Coh(\mathbb{X}_{\mathbf{p},\underline{\lambda}})$,
it can be decomposed as $\mathscr{M}_{t}\oplus\mathscr{M}_{f}$ where
$\mathscr{M}_{t}\in\mathscr{T}$ and $\mathscr{M}_{f}\in\mathscr{F}$
and
$\Hom(\mathscr{M}_{t},\mathscr{M}_{f})=\Ext^{1}(\mathscr{M}_{f},\mathscr{M}_{t})=0$
for any $\mathscr{M}_{t}\in\mathscr{T}$ and
$\mathscr{M}_{f}\in\mathscr{F}$.

The category $\mathscr{T}$ decomposes as a coproduct
$\mathscr{T}=\coprod_{x\in\mathbb{X}_{\mathbf{p},\underline{\lambda}}}\mathscr{T}_{x}$,
where $\mathscr{T}_{x}$ is equivalent to the category
$\rep_{0}(C_{r_x})$ consisting of nilpotent representations of the
cyclic quiver with $r_x$ vertices, where $r_x=p_{i}$ if
$x=\lambda_i$, $1\leq i\leq n$, and $r_x=1$ otherwise.

The category $\mathscr{F}$ has a filtration by objects of the form
$\mo(\vec{x})$, where $\vec{x}\in
L(\mathbf{p})=\mathbb{Z}\vec{x}_{1}\oplus\mathbb{Z}\vec{x}_{2}\oplus\cdots\oplus\mathbb{Z}\vec{x}_{n}/J$
where $J$ is the submodule generated by
$\{p_{1}\vec{x}_{1}-p_{s}\vec{x}_{s}|s=2,\cdots,n\}$. Set
$\vec{c}=p_{1}\vec{x}_{1}=\cdots=p_{n}\vec{x}_{n}\in L(\mathbf{p})$.
For $\mo(r\vec{c})$, there is a unique simple objects $S_{i,0}$ in
each $\mathscr{T}_{\lambda_{i}}$ with $\dim
\Hom(\mo(r\vec{c}),S)=1$. The simple objects are $S_a$ $(a\in
\mathbb{P}^{1}\backslash \underline{\lambda})$ and $S_{i,j}$ $(1\leq
i\leq n, 0\leq j\leq p_i-1)$, which satisfy
 the relations $\dim \Ext(S_{i,j},S_{i,j-1})=1$.

\subsection{Star-shaped graph and  loop algebra }\label{star-shaped graph}
Associating to the weight type $(\mathbf{p},\underline{\lambda})$,
we have a star-shaped graph $\Gamma$:\\

\         
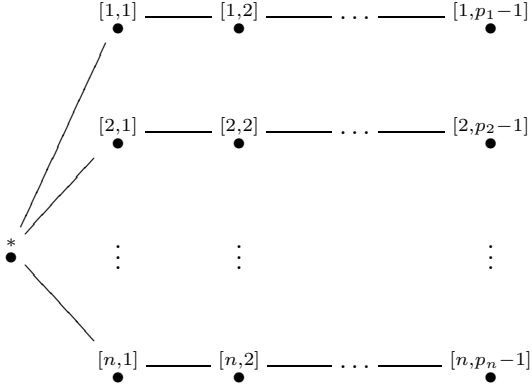
\xymatrix{&\stackrel{[1,1]}{\bullet}\rline & \stackrel{[1,2]}{\bullet}\rline & \ldots\rline &\stackrel{[1,p_1-1]}{\bullet}\\
&\stackrel{[2,1]}{\bullet}\rline & \stackrel{[2,2]}{\bullet}\rline & \ldots\rline &\stackrel{[2,p_2-1]}{\bullet}\\
\stackrel{\ast}{\bullet}\uurline\drline\urline&\vdots&\vdots& &\vdots\\
&\stackrel{[n,1]}{\bullet}\rline & \stackrel{[n,2]}{\bullet}\rline &
\ldots\rline &\stackrel{[n,p_n-1]}{\bullet} }\\
whose vertex set $\mci$ consists of the central vertex $\ast$ and
vertices in $n$ branches which are denoted by $[i,j]$, $1\leq i\leq
n$, $1\leq j\leq p_{i}-1$.

Consider the Kac-Moody algebra $\mfk{g}=\mfk{g}(\Gamma)$ associated
to the graph $\Gamma$. We have the \emph{loop algebra} of $\mfkg$,
denoted by $\mathcal{L}\mfkg$, which is defined to be the complex
Lie algebra generated by $h_{i,k}, e_{i,k}, f_{i,k}:i\in\mci,
k\in\mathbb{Z} $ and $c$ subject to the following relations:
\begin{align*}
&[h_{i,k},h_{j,l}]=k\delta_{k,-l}a_{ij}c,\\
&[e_{i,k},f_{j,l}]=\delta_{i,j}h_{i,k+l}+k\delta_{k,-l}c,\\
&[h_{i,k},e_{j,l}]=a_{ij}e_{j,l+k},\ [h_{i,k},f_{j,l}]=-a_{ij}f_{j,l+k},\\
&[e_{i,k},e_{i,l}]=0,\ [f_{i,k},f_{i,l}]=0,\ \  c\ \  central\\
&[e_{i,k_1},[e_{i,k_2},[\ldots,[e_{i,k_n},e_{j,l}]\ldots]=0,\
\text{for}\ n=1-a_{ij},\\
&[f_{i,k_1},[f_{i,k_2},[\ldots,[f_{i,k_n},f_{j,l}]\ldots]=0,\
\text{for}\ n=1-a_{ij}.
\end{align*}
 The root
systems of $\mfk{g}$ and $\mathcal{L}\mfk{g}$ are denoted by
$\Delta$ and $\hat{\Delta}$ respectively and the root lattices are
denoted by  $Q$ and $\hat{Q}=Q\oplus\mathbb{Z}\delta$. In view of
the graph $\Gamma$, the simple roots in $\Delta$ are denoted by
$\alpha_{\ast}$ and $\alpha_{ij}$ for $1\leq i\leq n$ and $1\leq
j\leq p_{i}-1$. We also know that
$\hat{\Delta}=\mathbb{Z}^{\ast}\delta\cup(\Delta+\mathbb{Z}\delta)$.

There is a natural identification of $\mathbb{Z}$-modules
$K_{0}(\Coh(\mathbb{X}))\cong\hat{Q}$ given by

$[S_{i,j}]\mapsto\alpha_{ij},\ \text{for}\ j=1,\cdots,p_{i}-1,\ \
[S_{i,0}]\mapsto\delta-\sum_{j=1}^{p_{i}-1}\alpha_{ij},\ \
[\mo(k\vec{c})]\mapsto\alpha_{\ast}+k\delta$.

Naturally, the non-negative combinations of the elements
$\alpha_{ij}$, $\delta-\sum_{j=1}^{p_{i}-1}\alpha_{ij}$,
$\alpha_{\ast}+k\delta$ and $\delta$ form the positive cone
$\hat{Q}_+$.

\subsection{Derived equivalence and the Lie algebra}
\label{Lie algebra} In \cite{R2}, Ringel introduced the class of
canonical algebras attached to $(\mathbf{p},\underline{\lambda})$.
It is well known that there is a triangle equivalence
$D^b(\Coh(\mathbb{X}_{\mathbf{p},\underline{\lambda}})) \simeq
D^b(\modl(\Lambda_{\mathbf{p},\underline{\lambda}}^{op}))$ where
$\Coh(\mathbb{X}_{\mathbf{p},\underline{\lambda}})$ is a hereditary
abelian category. Therefore, their root categories are equivalent.
We simply write $\Lambda$ for
$\Lambda_{\mathbf{p},\underline{\lambda}}^{op}$. Then by \ref{lie
algebra}, we can define a $\hat{Q}$-graded complex Lie algebra $L$
on the root category of $\Lambda$.

The set of indecomposable objects of
$\mathcal{R}_{\mathbf{p},\underline{\lambda}}=D^b(\Coh(\mathbb{X}_{\mathbf{p},\underline{\lambda}}))/(T^{2})$
is $\ind\mathcal{R}_{\mathbf{p},\underline{\lambda}}=
(\ind\Coh(\mathbb{X}_{\mathbf{p},\underline{\lambda}})\bigcup
\{TY|Y\in \ind\Coh(\mathbb{X}_{\mathbf{p},\underline{\lambda}})\}.$
For any simple object $S$, $S[r]$ denotes the unique object $S[r]$
with length $r$ and top $S$ for $r>0$,
 and denotes the unique object $TY$ for $r<0$, where $Y$ is of length $-r$ with $\Ext^1(Y,S)\neq 0$.
$H_r$ is the set of $X\in
\ind\mathcal{R}_{\mathbf{p},\underline{\lambda}}$ of type $r\delta$
and with $\Hom(X,S_{i,j})=0$ for all $1\leq i\leq n$ and $1\leq j
\leq p_i-1$.

\begin{Lem}

(i) For any $X\in \ind\mathcal{R}_{\mathbf{p},\underline{\lambda}}$,
the image of $X$ in the root category of the canonical algebra
$\Lambda$ is denoted by $F(X)$. Assume $F(X)\in
\mathcal{P}_2(\Lambda,\underline{e})$,
$\hat{1}_{G_{\underline{e}}F(X)}$ is the equivalence class of the
characteristic function of the orbit $G_{\underline{e}}F(X)$. Then
$\hat{1}_{G_{\underline{e}}F(X)}\in I_{GT}(\dimv\ \underline{e})$

(ii) The set $F(H_r)\subset
\mathcal{P}_2(\Lambda,\underline{e(r)})$, and
$\hat{1}_{G_{\underline{e(r)}}F(H_r)}\in I_{GT}(\dimv\
\underline{e(r)})$. Moreover,
$\chi(G_{\underline{e(r)}}F(H_r)/G_{\underline{e(r)}})=2$.
\end{Lem}
\begin{proof}
(i) is trivial because $F(X)$ is also indecomposable in the root category of $\Lambda$.\\
(ii) The Serre subcategory generated by $\mo(k\vec{c})$ for $k\in
\mathbb{Z}$ , $S_a[l] (a\in \mathbb{P}^{1}\backslash
\underline{\lambda},l\geq 1)$ and $S_{i,0}[lp_i]$ $(1\leq i\leq
n,l\geq 1)$ is equivalent to the category $\Coh(\mathbb{P}^{1})$.
Therefore, it is enough to prove the non-weighted case. We have
$D^b(\Coh(\mathbb{P}^{1}))=D^b(\rep\overrightarrow{Q})$, where
$\overrightarrow{Q}$ is the Kronecker quiver. There exists
$\underline{e(r)}$ such that $F(H_r)\subset
\mathcal{P}_2(\Lambda,\underline{e(r)})$.
 The results in the Kronecker quiver case imply  $\hat{1}_{G_{\underline{e(r)}}F(H_r)}\in I_{GT}(\dimv\ \underline{e(r)})$
 and $\chi(G_{\underline{e(r)}}F(H_r)/G_{\underline{e(r)}})=2$.
\end{proof}

\section{new proof}
\subsection{Main result}
\begin{Thm}\label{Main Thm}
If $\mathbb{X}_{\mathbf{p},\underline{\lambda}}$ is a weighted
projective line over the complex field $\mathbb{C}$ and
$\alpha\in\hat{Q}$, there is an indecomposable sheaf in
$\Coh(\mathbb{X}_{\mathbf{p},\underline{\lambda}})$ of type $\alpha$
if and only if $\alpha$ is a positive root. Moreover, there is a
unique indecomposable for a real root, infinitely many for an
imaginary root.
\end{Thm}

This theorem is proved in \cite{W} over any algebraically closed
field. In the case of the complex field $\mathbb{C}$, we find a new
proof as follows, which also uses the Hall algebras. We define a
$\hat{Q}$-graded complex Lie algebra $L$ on the root category
 $\mathcal{R}_{\mathbf{p},\underline{\lambda}}$ (section \ref{Lie algebra})
and there is a subalgebra satisfying the relations of the loop
algebra.

Set $l(r)=1,for\ \  r\geq 0$ and $l(r)=-1,for\ \  r<0$. For any
$X\in \ind\mathcal{R}_{\mathbf{p},\underline{\lambda}}$, we write
$\hat{1}_{(X)}=\hat{1}_{G_{\underline{e}}F(X)}$
 and  $\hat{1}_{(H_r)}=\hat{1}_{G_{\underline{e(r)}}F(H_r)}$ for short.

\begin{Thm}
\label{Thm2} The following elements  satisfy the relations in
$\mathcal{L}\mfkg$.
\[
e_{v,r}=
\begin{cases}
l(r)\hat{1}_{(S_{i,j}[rp_i+1])} &v=[i,j]\\
l(r)\hat{1}_{(\mo(r\vec{c}))} &v=*
\end{cases}
\quad f_{v,r}=
\begin{cases}
l(r-1)\hat{1}_{(S_{i,j-1}[rp_i-1])} &v=[i,j]\\
l(r)\hat{1}_{(\mo(-r\vec{c}))} &v=*
\end{cases}
\]
\[
c=-\delta\quad h_{v,r}=
\begin{cases}
-\alpha_v& r=0\\
l(r)\hat{1}_{(S_{i,j}[rp_i])}-l(r)\hat{1}_{(S_{i,j-1}[rp_i])}  &r\neq 0, v=[i,j]\\
l(r)\hat{1}_{(H_r)}&r\neq 0, v=*
\end{cases}
\]

\end{Thm}

\subsection{Proof of Theorem \ref{Thm2}}\label{relation}

We note that
$[\hat{1}_{\mathcal{O}_1},\hat{1}_{\mathcal{O}_2}](M)=0$ for $M$
decomposable and the triangles $X\rightarrow Y\rightarrow
Z\rightarrow$ with $X,Y,Z\in
\ind\mathcal{R}_{\mathbf{p},\underline{\lambda}}$ are in 1-1
correspondence with short exact sequences in
$\Coh(\mathbb{X}_{\mathbf{p},\underline{\lambda}})$. The section 3
of \cite{W} is still true for the complex field. However, we
calculate the Euler characteristics instead of counting numbers.

(i)\begin{equation*}[l(r)\hat{1}_{(S_{i,j}[r])},
l(s)\hat{1}_{(S_{i,k}[s])}]= \begin{cases}
\delta_{j-r,k}l(r+s)\hat{1}_{(S_{i,j}[r+s])}-\delta_{j,k-s}l(r+s)\hat{1}_{(S_{i,k}[r+s])}& r+s\neq 0\\
-\delta_{j-r,k}[S_{i,j}[r]]& r+s=0
\end{cases}
\end{equation*}

Proof of (i): In one tube, if $0\rightarrow X\rightarrow Y
\rightarrow Z\rightarrow 0$ is a short exact sequence of
indecomposable objects, then there is a unique short exact sequence
with the same terms up to automorphisms of any two of $X,Y,Z$. Using
the fact $\chi(\text{one point})=1$, we complete the proof.

Note that we can prove all relations  in one tube by (i) now.\\

(ii)
$[h_{*,r},h_{*,-r}]=[l(r)\hat{1}_{(H_r)},l(-r)\hat{1}_{(H_{-r})}]
=-r\delta
\chi(G_{\underline{e(r)}}H_r/G_{\underline{e(r)}})=-2r\delta=2rc$\\

(iii) For $[e_{*£¬r}, f_{*,s}]$,

 if $r+s=0$,
$[e_{*£¬r}, f_{*,-r}]=
-[\hat{1}_{(\mo(r\vec{c})},\hat{1}_{(\mo(-r\vec{c})}]
=-\chi((\mo(r\vec{c}))[\mo(r\vec{c})] =-[\mo(r\vec{c})] =h_{*,0}+rc$

if $r+s\neq 0$, assume $r+s>0$,
 we get the short exact sequence
 $0\rightarrow \mo(-(r+s)\overrightarrow{c})\rightarrow \mo\rightarrow Y\rightarrow 0$
with $Y\in H_{r+s}$,
 $\dim\Hom(\mo, Y)=r+s$. The non-epimorphisms form a subspace of dimension $r+s-1$ and
  each short exact sequence is determined by an epimorphism up to an automorphism of $\mo(-(r+s)\overrightarrow{c})$.
 $[e_{*£¬r}, f_{*,s}](Y)=\chi(\mathbb{C}^{r+s-1})=1$.
 That implies $[e_{*£¬r}, f_{*,s}]=h_{*,r+s}$.\\

 (iv)We assume $r>0$,
 The support of the function $[h_{[i,1],r},e_{*,s}]$ is the orbit of $\mo((r+s)\overrightarrow{c})$.
   For $X\in(\mo((r+s)\overrightarrow{c}))$,
 $[h_{[i,1],r},e_{*,s}](X)=-\chi(\text{one point})=-1$, then
 $[h_{[i,1],r},e_{*,s}]=-e_{*,r+s}$.\\

 (v)We assume $r>0$,
 The support of the function $[h_{*,r},e_{*,s}]$ is the orbit of $\mo((r+s)\overrightarrow{c})$.
   For $X\in(\mo((r+s)\overrightarrow{c}))$,
 $[h_{*,r},e_{*,s}](X)=\chi(\mathbb{P}^{1})=2$, then
 $[h_{*,r},e_{*,s}]=2e_{*,r+s}$.
 $\hfill\blacksquare$

\subsection{Proof of Theorem \ref{Main Thm}}

$L$ is a $\hat{Q}$-graded complex Lie algebra with
$L_0=\hat{Q}\bigotimes_{\mathbb{Z}}\mathbb{C}$. For
$\phi\in\hat{Q}_+$, if there is an indecomposable sheaf $X$ in
$\Coh(\mathbb{X}_{\mathbf{p},\underline{\lambda}})$ of type $\phi$,
then $\hat{1}_{(X)}\in L_\phi$ and $L_\phi\neq0$. If there is no
indecomposable sheaf of type $\phi$, $L_\phi=0$. The case of
$-\phi\in\hat{Q}_+$ is similar.

For $\phi\in\hat{Q}_+$, we want to determine whether or not
$L_\phi=0$. We need the following two lemmas:

\begin{Lem}\label{nilpotent}
Let $v$ be a vertex of the star-shaped graph. The operators $\ad\
e_{v,0}$ and $\ad\ f_{v,0}$ are locally nilpotent.
\end{Lem}

\begin{proof}

For any $\psi\in\hat{Q}$ and $f\in L_\psi$, we need to show $(\ad\
e_{v,0})^n(f)=(\ad\ f_{v,0})^n(f)=0$, for some $n$. It is enough to
prove $(\ad\ \hat{1}_{X})^n(\hat{1}_{Y})=0$ for any two
indecomposable sheaves $X,Y$ with $\Ext^1(X,X)=0$:

If $Z$ is in the support of $(\ad\ \hat{1}_{X})(\hat{1}_{Y})$, then
$Z$ is the middle term of a nonsplit exact sequence whose end terms
are $X$ and $Y$, so

$\dim\Ext^1(X,Z)+\dim\Ext^1(Z,X)<\dim\Ext^1(X,Y)+\dim\Ext^1(Y,X)$,
thus

 $(\ad\ \hat{1}_{X})^n(\hat{1}_{Y})=0$ for
 $n>\dim\Ext^1(X,Y)+\dim\Ext^1(Y,X)$.
\end{proof}

\begin{Lem}
Let $v$ be a vertex of the star-shaped graph and write $\alpha_v$
for the simple root corresponding to $v$.  For any
$\phi\in\hat{Q}_+$, we have $L_{\phi}\simeq L_{s_v(\phi)}$.
\end{Lem}
\begin{proof}
As proved in \ref{relation}, $e_{v,0}\in L_{\alpha_v}$ and
$f_{v,0}\in L_{-\alpha_v}$ satisfy $[e_{v,0},f_{v,0}]=h_{v,0}$ and
for $f\in L_\psi$, $(\ad\ h_{v,0})(f)=(\alpha_v,\psi)f$. From Lemma
\ref{nilpotent}, $\ad\ e_{v,0}$ and $\ad\ f_{v,0}$ are locally
nilpotent. So the operator $\theta=exp(\ad\ e_{v,0})exp(-\ad\
f_{v,0})exp(\ad\ e_{v,0})$ acts on $h_{v,0}$ as multiplication by
$-1$. For $f\in L_\phi$, we have
$\theta(f)=\Sigma_{r\in\mathbb{Z}}f'_r$ with $f'_r\in
L_{\phi+r\alpha_v}$.
$$\Sigma_{r\in\mathbb{Z}}(\alpha_v,\phi)f'_r=\theta([h_{v,0},f])=[\theta(h_{v,0}),\theta(f)]=[-h_{v,0},\theta(f)]$$
$$=[-h_{v,0},\Sigma_{r\in\mathbb{Z}}f'_r]=-\Sigma_{r\in\mathbb{Z}}(\alpha_v,\phi+r\alpha_v)f'_r
$$

Comparing the coefficients of the above equation, we get
$\theta(f)=f'_r$ with $r=-(\alpha_v,\phi)$, which means
$\theta(L_{\phi})\subseteq L_{\phi-(\alpha_v,\phi)\alpha_v}$.
Similarly $\theta^{-1}(L_{\phi-(\alpha_v,\phi)\alpha_v})\subseteq
L_{\phi}$. Thus the operator $\theta=exp(\ad\ e_{v,0})exp(-\ad\
f_{v,0})exp(\ad\ e_{v,0})$ induces an isomorphism $L_{\phi}\simeq
L_{s_v(\phi)}$.
\end{proof}

For $\phi\in\hat{Q}$, we can
 reduce to the following three cases by a sequence of
reflections:

$\pm \alpha_v+r\delta$;

 $\alpha+r\delta$, with $\alpha$ in the fundamental region;

$\alpha+r\delta$, where $\alpha$ is not positive or negative, or has
disconnected support.

For the first case: $\dim L_{\phi}=\dim L_{\pm \alpha_v+r\delta}=1$,
there is a unique indecomposable sheaf;

the second case: $\dim L_{\phi}=\dim L_{\alpha+r\delta}=\infty$,
there are infinitely many indecomposable sheaves (see
 \cite{W2});

the last case: $\dim L_{\phi}=\dim L_{\alpha+r\delta}=0$, there is
no indecomposable object.

\end{document}